\documentclass{amsart}

\usepackage{amsmath,amssymb,amsthm}

\newtheorem{theorem}{Theorem}[section]
\newtheorem{lemma}[theorem]{Lemma}

\theoremstyle{definition}
\newtheorem{example}[theorem]{Example}
\newtheorem*{main theorem}{Main Result}

\newtheorem{problem}[theorem]{Open Problem}

\DeclareMathOperator{\relrho}{\mathnormal{\rho}}

\title[Finitely presented infinite congruence-free monoids]{A countable family of finitely presented infinite congruence-free monoids}

\author{Alan Cain}

\address{Centro de Matematica, Universidade do Porto, Rua do Campo Alegre, 687 4169–007, Porto, PORTUGAL}

\email{\texttt{ajcain@fc.up.pt}}

\author{Victor Maltcev}

\address{Department of Mathematics and Statistics, Sultan Qaboos University, Al-Khodh 123, Muscat, 
Sultanate of OMAN}

\email{\texttt{victor.maltcev@gmail.com}}

\author{Abdullahi Umar}

\address{Department of Mathematics and Statistics, Sultan Qaboos University, Al-Khodh 123, Muscat, 
Sultanate of OMAN}

\email{\texttt{aumarh@squ.edu.om}}

\begin{document}

\begin{abstract}
We prove that monoids
\begin{align*}
\mathrm{Mon}\langle a,b,c,d :\;& a^nb=0,~ac=1,~db=1,~dc=1\\
& dab=1,~da^2b=1,\ldots,~da^{n-1}b=1\rangle
\end{align*}
are congruence-free for all $n\geq 1$. This provides a new countable family of finitely presented congruence-free monoids, bringing one step closer to understanding the Boone--Higman Conjecture. We also provide examples which show that finitely presented congruence-free monoids may have quadratic Dehn function. 
\end{abstract}


\maketitle

\section{Introduction and Main Result}
\label{sec:intro}

It is a classical theorem of Rabin (see~\cite{LS}) that every
countable group embeds in a finitely generated simple group. Taking a
deeper look at this question, Boone and Higman proved in~\cite{Boone},
see also~\cite{LS}, that a finitely generated group has soluble word
problem if and only if it embeds in a simple subgroup of a finitely
presented group. This motivated them to raise the question, which is
now still an open problem, referred to as the Boone--Higman
Conjecture:

\begin{problem}
Does every group with soluble word problem embed in a finitely
presented simple group?
\end{problem}

Note that the condition of having soluble word problem is crucial as
every finitely presented simple group necessarily has soluble word
problem. There are plenty of results regarding finitely presented
simple groups, out of which we will mention only the result of
Rover~\cite{Rover} that Grigorchuk groups embed in finitely presented
simple groups, and a result of Scott~\cite{Scott} that there is a
finitely presented simple group with insoluble conjugacy problem. Both
results of Rover and Scott rely on the infinite family of finitely
presented simple groups found by Higman~\cite{Higman}.

Strangely, the analogous questions for semigroups have not been
studied so extensively. The natural counterpart of simplicity for
Semigroup Theory to take is \emph{congruence-freeness}: recall that a
semigroup $S$ is congruence-free if it has only two congruences -- the
identity relation and the relation $S\times S$. First of all, the
analogue of Rabin's Theorem does hold, as was shown by Byleen
in~\cite{Byleen}. Secondly, in a series of paper by
Birget~\cite{Birget1}--\cite{Birget5}, it was developed and studied
the analogue of Higman's countable family.

Recent work by the second author has exhibited a countable family of
finitely presented bisimple $\mathcal{H}$-trivial congruence-free
monoids~\cite{Victor1}, and proved that every finite semigroup embeds
in a finitely presented congruence-free monoid~\cite{Victor2}. These
two papers contained the only known examples of finitely presented
infinite congruence-free monoids that are not groups.

The main goal of this note is to further expand the class of known
examples of finitely presented infinite congruence-free non-group
monoids by proving the following result:
\begin{main theorem}
The monoids
\begin{eqnarray*}
M_n=\mathrm{Mon}\langle a,b,c,d &:& a^nb=0,~ac=1,~db=1,~dc=1\\
&& dab=1,~da^2b=1,\ldots,~da^{n-1}b=1\rangle
\end{eqnarray*}
are congruence-free for all $n\geq 1$.
\end{main theorem}
We obtained this family while trying to embed monoids
$\mathrm{Mon}\langle a,b:a^nb=0\rangle$ in finitely presented
congruence-free monoids. If one increases the exponent of $b$ by $1$,
the question of embedding seems to become much harder, and we have
been unable to resolve it even for the monoid $\mathrm{Mon}\langle
a,b:a^2b^2=0\rangle$. Thus we ask the following questions:

\begin{problem}
Does the monoid $\mathrm{Mon}\langle a,b:a^2b^2=0\rangle$ embed in a
finitely presented congruence-free monoid? If `yes', can one write an
explicit presentation for the monoid to contain $\mathrm{Mon}\langle
a,b:a^2b^2=0\rangle$?
\end{problem}

Before we embark on the proof of our theorem, we provide all the ingredients required for the proof.

\section{Preliminaries}\label{sec:preliminaries}

We will require some information from both semigroup theory and computer science.

The correspondence between normal groups and homomorphisms in group
theory is paralleled by the correspondence between \emph{congruences}
and homomorphisms in semigroup theory. For a semigroup $S$, a binary
relation $\relrho\subseteq S\times S$ is called a \emph{congruence} if
it is an equivalence relation and compatible with multiplication on
the left and right: that is, if $x\relrho y$ for some $x,y\in S$, then
$zx\relrho zy$ and $xz\relrho yz$ for all $z\in S$. The equivalence
classes of $S$ with respect to a congruence $\relrho$ on it, form a
factor-semigroup denoted by $S/\!\relrho$. A subset $I\subseteq S$ is
called an \emph{ideal} of a semigroup $S$, if $IS\cup SI\subseteq
I$. With every ideal $I\subseteq S$ there is associated the so-called
\emph{Rees congruence} $\relrho_I=(I\times I)\cup\Delta$, where
$\Delta$ is the identity relation on $S$. A semigroup is called
\emph{simple} if it has only one ideal, namely the whole semigroup
itself. A semigroup $S$ with zero $0$ is called \emph{$0$-simple} if
it has only two ideals, namely the whole semigroup and
$\{0\}$. Because of the Rees congruences, one easily sees that every
congruence-free semigroup is either simple or $0$-simple. For further
background on semigroups, see~\cite{Howie}.

A \emph{rewriting system} $(A,R)$ comprises a finite alphabet $A$ and
a subset $R\subseteq A^{\ast}\times A^{\ast}$, where $A^{\ast}$ stands
for the free monoid over $A$. Every pair $(l,r)$ from $R$ is called a
\emph{rule} and normally is written as $l\to r$. For $x,y\in A^{\ast}$
we write $x\to y$, if there exist $\alpha,\beta\in A^{\ast}$ and a
rule $l\to r$ from $R$ such that $x=\alpha l\beta$ and $y=\alpha
r\beta$. Denote by $\to^{\ast}$ the transitive reflexive closure of
$\to$. A rewriting system $(A,R)$ is:
\begin{itemize}
\item
\emph{confluent} if for every words $w,x,y\in A^{\ast}$ such that $w\to^{\ast}x$ and $w\to^{\ast}y$, there exists $W\in A^{\ast}$ such that $x\to^{\ast}W$ and $y\to^{\ast}W$;
\item
\emph{terminating} if every infinite derivation $x_0\to x_1\to x_2\to\cdots$ stabilises.
\end{itemize}
Confluent terminating systems, which are also called \emph{complete
  systems}, give a very convenient way of working with finitely
generated monoids: if a monoid is presented by $M=\mathrm{Mon}\langle
A:l_i=r_i\quad i\in I\rangle$ and it turns that the system
$S=(A,\{l_i\to r_i\}_{i\in I})$ is complete, then the elements of $M$
are in bijection with the \emph{normal forms} for $S$, i.e. those
words from $A^{\ast}$ which omit the subwords $l_i$, and for a word
$w\in A^{\ast}$ to find its normal form with respect to $S$, we just
need to apply $\to$ successively to $w$ as many times as we can (this
process must stop by the termination condition) and the result will
always be the same word depending only on $w$. See~\cite{Otto} for
more details on rewriting systems.

Our final notation to fix is as follows: if $A$ is a finite generating
set for a semigroup $S$ and $u$ and $v$ are words over $A$, then by
$u\equiv v$ we will mean that $u$ and $v$ coincide graphically; and by
$u=v$ we will mean that $u$ and $v$ represent the same element of the
semigroup $S$.

\section{Proof of the Main Result}\label{sec:proof}

One easily sees that the presentation for $M_n$ considered as the corresponding rewriting system
is complete. Thus we can use the normal forms for $M_n$ with respect to this complete system.
Note that if a word in the normal form contains $d$, then all the
letters following this distinguished $d$ are only $a$'s and $d$'s. We will use this fact quite frequently in the proof.

Let us start collecting some information about $M_n$:
\begin{lemma}
$M_n$ is $0$-simple.
\end{lemma}

\begin{proof}
To prove the lemma, it suffices to show that if $w$ is a non-zero word over $\{a,b,c,d\}^{\ast}$ in its normal form, then $M_nwM_n=M_n$. We will prove it by induction on the length $|w|$ of $w$. The base case $|w|=0$, i.e. $w=1$, is trivial. Now let us do the transition $(<|w|)\mapsto |w|$.

Let us first assume that $w$ contains letters $d$. We can take the last letter $d$, after which by the above remark there can follow only letters $a$. Hence $w$ is representable as $w\equiv w'da^k$ for some $k\geq 0$. Then $wc^kab=w'dab=w'$ and we may apply induction.

So, let now $w\in\{a,b,c\}^{\ast}$. If $w$ starts with $b$ or $c$, then by premultiplying $w$ with $d$, we can cancel out that corresponding $b$ or $c$, and then use induction. So, we may assume that $w$ starts with $a$. If $w$ is a power of $a$, then by the relation $ac=1$ we immediately get $M_n=M_nwM_n$. So, again because of the relation $ac=1$, we may assume that $w\equiv a^kbw'$ for some $k\geq 1$. Because of the relation $a^nb=0$, we see that $k<n$. But then $da^kb=1$ and so $dw=w'$ and we are done by induction.
\end{proof}

In order to prove that $M_n$ is congruence-free, we proceed by
induction on $|u|+|v|$ proving that if $\relrho$ is a congruence on
$M_n$ and $u\relrho v$ for some distinct normal form words $u$ and $v$
over $\{a,b,c,d\}$, then $\relrho=M_n\times M_n$.

Let us first check the base case -- without loss we may assume that
$|u|=0$ and $|v|=1$. Then we have that $u=1$ and
$v\in\{a,b,c,d,0\}$. Having $1\relrho 0$ immediately implies
$\relrho=M_n\times M_n$. If $1\relrho a$, then $a^{n-1}b\relrho
a^nb=0$ and so $0\relrho da^{n-1}b=1$. If $1\relrho b$ or $1\relrho
c$, then $1=db=dc\relrho d$ and so $ab\relrho dab=1$, implying
$0=a^nb\relrho a^{n-1}$, which yields $0\relrho a^{n-1}c^{n-1}=1$. So,
in any of the cases we obtain $1\relrho 0$ and so the base case holds.

Now we do the step $(<|u|+|v|)\mapsto (|u|+|v|)$.

Let us first sort out the case when both $u$ and $v$ contain $d$,
i.e. $u\equiv Uda^p$ and $v\equiv Vda^q$ for some $p,q\geq 0$. If
$p=q$, then $U\not\equiv V$ and $U=Uda^pc^{p+1}\relrho
Vda^qc^{q+1}=V$ and we may use induction. So, let, say, $p>q$.
Then $Uda^{p-q}\relrho Vd$ and so $0=Uda^{(p-q)+n-1}b\relrho Vda^{n-1}b=V$, hence we may use
$0$-simplicity to conclude that $0\relrho 1$ and consequently $\relrho=M_n\times M_n$.

Now let us deal with the case when only one of $u$ and $v$ contains
$d$: say $u\equiv Uda^p$ and $v\in\{a,b,c\}^{\ast}$ for some $p\geq
0$. Then $Ud=uc^p\relrho vc^p\in\{a,b,c\}^{\ast}$. Let $v'$ be the
normal form for $vc^p$. If $v'$ has $a$ as the last letter, then
$U=Uda^{n-1}b\relrho v'a^{n-1}b=0$ and we may use $0$-simplicity. So,
we may assume that $v'$ does not end with $a$. Since $vc^p$ does not
contain $d$'s and rewriting does not introduce letters $d$, and in
normal forms $c$ cannot follow $a$, one sees now that if $v'$ contains
$a$'s, then each such letter $a$ in $v'$ is a part of a subword $a^kb$
with $1\leq k\leq n-1$. Using this fact and the relations $dc=1$ and
$da^{k}b=1$ for all $1\leq k\leq n-1$, there exists an appropriate
$m\geq 0$ such that $d^mv'=1$. Then $d^mUd\relrho 1$. In particular,
recalling that $db=1$, $d$ is invertible in $M_n/\!\relrho$, and so
$ab$ is invertible in $M_n/\!\relrho$, which yields from $a^nb=0$ that
$a^n=0$ in $M_n/\!\relrho$. Then $1=a^nc^n=0$ in $M_n/\!\relrho$ and
so $\relrho=M_n\times M_n$.

So, from now on we may assume that $u,v\in\{a,b,c\}^{\ast}$.

Let us first deal with the case when one of $u$ and $v$ is empty,
say $v\equiv 1$. Then $u\not\equiv 1$. If $u$ starts with $b$ or
$c$, then, since we already know from the presentation that $b$ and $c$ are left invertible in $M_n$, respectively $b$ or $c$ is invertible in $M_n/\!\relrho$, and
then $d$ is invertible in $M_n/\!\relrho$, yielding $1=0$ in $M_n/\!\relrho$ as
above. So assume, $u\equiv a^kU$, $k\geq 1$ and such that either $U\equiv
1$ or $U$ starts with $b$. If $U\equiv 1$, then $a$ is invertible
in $M_n/\!\relrho$, which yields $b=0$ in $M_n/\!\relrho$ and so $1=db=0$ in
$M_n/\!\relrho$. So let $U\not\equiv 1$. Then $0=a^{n-1}u\relrho a^{n-1}$ and so
$1=a^{n-1}c^{n-1}=0$ in $M_n/\!\relrho$.

Let now $u\not\equiv 1$ and $v\not\equiv 1$. Assume first that at
least one of $u$ and $v$ starts with $a$, say, $u\equiv a^kU$ for some $k\geq 1$. We may also assume that $k$ is a maximal possible number with $u\equiv a^kU$. Then
either $U\equiv 1$, or $U$ starts with $b$.
\begin{itemize}
\item
$U\equiv 1$. Then $va^{n-1}b\relrho a^{k+n-1}b=0$. If $va^{n-1}b\neq 0$, we may use
$0$-simplicity. If $va^{n-1}b=0$, then $v$ must end with $a$. But then
$v\equiv Va$ and so $a^{k-1}=uc\relrho vc=V$ and since $a^{k-1}\not\equiv V$, we
may use induction.
\item
$U\equiv bU_1$ for some
$U_1\in\{a,b,c\}^{\ast}$. Then $k\leq n-1$ and $u\equiv a^kbU_1\relrho v$. Then $a^{n-k}v\relrho
0$. Again, if $a^{n-k}v\neq 0$, then we may use $0$-simplicity. If
$a^{n-k}v=0$, then $v$ must start with $a^kb$: $v\equiv a^kbV$ and then
$U_1=du\relrho dv=V$ and again since $U_1\not\equiv V$, we may use
induction.
\end{itemize}
Finally, let neither of $u$ and $v$ start with $a$. If $u$ and $v$
start with the same letter $x$ (which will be either $b$ or $c$) and
$u\equiv xU$ and $v\equiv xV$, then $U\not\equiv V$ and $U=du\relrho
dv=V$ and we may use induction. So, without loss we will assume that
$u\equiv bU$ and $v\equiv cV$. Then $U=dbU\relrho dcV=V$, so we may
assume that $U\equiv V$ (otherwise use induction). But also
$U=dabU\relrho dacV=dV=dU$. We have that $U\neq dU$ and $|U|+|dU|\leq
2|U|+1<2|U|+2=|u|+|v|$, and thus may use induction.

\section{Dehn function}

All the so far known examples of finitely presented congruence-free
monoids -- from the Main Result, and from~\cite{Victor1}
and~\cite{Victor2} -- admit finite complete length-decreasing
rewriting systems, and thus have linear Dehn functions. The following
example shows that finitely presented congruence-free monoid may have
quadratic Dehn function.

\begin{example}
The monoid $M$ presented by the finite complete system
\begin{eqnarray*}
ab &\to& ba\\
cbad &\to& 1\\
cb^2 &\to& 1\\
a^2d &\to& 1\\
cad &\to& 0\\
cbd &\to& 0\\
cd &\to& 1
\end{eqnarray*}
is congruence-free and has quadratic Dehn function.
\end{example}

\begin{proof}
That the monoid has quadratic Dehn function is immediate.
It is a routine to check that $M$ is $0$-simple.

We proceed by induction on $|u|+|v|$ proving that if $u$ and
$v$ are in their normal forms and $u\relrho v$ for some congruence
$\relrho$ on $M$, then $\relrho=M\times M$. The base case is obvious.
Now we do the step $(<|u|+|v|)\mapsto (|u|+|v|)$.

First deal with the case when both $u$ and $v$ contain $c$. Then $u$
and $v$ decompose as $u\equiv Ucb^pa^q$ and $v\equiv Vcb^ra^s$ where
$p,q,r,s\geq 0$. Since $a$ is right cancellative, we may assume that
either $q=0$ or $s=0$. Without loss we will assume that $s=0$,
i.e. $v\equiv Vcb^r$.

First we consider the case when $q=0$. If $p=r$, then since
$p,r\in\{0,1\}$, by $cd=1$ and $cbad=1$, we may use induction. If
$p=1$ and $r=0$, i.e. $Ucb\relrho Vc$ and so $U=Ucbad\relrho Vcad=0$
and we may use $0$-simplicity. The case when $p=0$ and $r=1$ is dealt
similarly.

Thus we may assume that $q>0$. To recall: $Ucb^pa^q\relrho Vcb^r$. We
will go through four cases depending whether $p$ and $r$ are $0$ or
$1$:
\begin{itemize}
\item $p=r=0$: $Uca^q\relrho Vc$. If $q=1$, then $Uca\relrho Vc$ and
  so $0=Ucad\relrho Vcd=V$ and we may use $0$-simplicity. So, we may
  assume that $q\geq 2$. Then $Uca^{q-2}=Uca^qd\relrho Vcd=V$ and so
  we may assume that $V\equiv Uca^{q-2}$. Thus, initially we had
  $Uca^q\relrho Uca^{q-2}c$. Then $Ua^q=Uca^qb^2\relrho
  Uca^{q-2}cb^2=Uca^{q-2}$ and now we may use induction.
\item
$p=0$ and $r=1$: $Uca^q\relrho Vcb$. Then $Ucba^q\relrho V$ and so we may assume that $V\equiv Ucba^q$. Thus initially we had $Uca^q\relrho Ucba^qcb$, and postmultiplying this with $ad$, we obtain $Uca^{q-1}\relrho Ucba^q$ and so $Uc\relrho Ucba$ and now we may use induction.
\item
$p=1$ and $r=0$: $Ucba^q\relrho Vc$. Then $Ucba^{q-1}=Ucba^q\cdot ad\relrho Vcad=0$ and we may use $0$-simplicity.
\item
$p=r=1$: $Ucba^q\relrho Vcb$. Then $0\neq Ucba^qd\relrho Vcbd=0$ and we may use $0$-simplicity.
\end{itemize}

Thus, from now on we may assume that $u$ and $v$ do not both contain $c$'s. Let us deal with the case when one of $u$ and $v$ contain $c$'s. Say, $u\equiv Ucb^pa^q$ and $v\in\{a,b,d\}^{\ast}$. Then $Ucb^p=u(ad)^q\relrho v(ad)^q$. Recall that $p\in\{0,1\}$. Now we have $0=Ucbd\relrho v(ad)^qb^{1-p}d\neq 0$ and we may use $0$-simplicity.

Therefore, we may assume that none of $u$ and $v$ contains $c$. By symmetry, we may assume that none of $u$ and $v$ contains $d$. Since $ab=ba$, $a$ is right cancellative, and $b$ is left cancellative, essentially we are left to deal with two cases:
\begin{itemize}
\item
$u=b^p$ and $v=a^q$. Without loss we will assume that $q>0$. Now, $b^{2p}\relrho a^{2q}$ and so $1\relrho c^pa^{2q}$. This means that $a$ is invertible in $M/\!\relrho$, and so $d$ is invertible in $M/\!\relrho$. Thus from $cad=0$, we have that $c=0$ in $M/\!\relrho$ and so $0=cb^2=1$ in $M/\!\relrho$.
\item
$u=b^pa^q$ and $v=1$. Again without loss we may assume
that $q>0$, hence $a$ is invertible in $M/\!\relrho$, and as in the previous case we deduce that $0=1$ in $M/\!\relrho$.
\end{itemize} 
\end{proof}

\begin{problem}
Characterise the Dehn functions of finitely presented congruence-free monoids.
\end{problem}

\section{Concluding remarks}

All the examples of finitely presented congruence-free monoids we have
met so far -- from this paper and from~\cite{Victor1}
and~\cite{Victor2} -- are not only simple or $0$-simple, but in fact
bisimple or $0$-bisimple. We have not managed to find an example of a
finitely presented congruence-free but not bisimple monoid and so
finish the paper with the following question:

\begin{problem}
Does there exist a finitely presented congruence-free non-bisimple monoid?
\end{problem}

\section*{Acknowledgements}

This paper was written while the visit of the first author to Sultan Qaboos University. We would like to thank SQU for hospitality.


\end{document}